\documentclass[12pt,a4paper]{article}

\usepackage{amsthm}
\usepackage{color}
\usepackage{amsmath}

\usepackage{amssymb}
\usepackage{epsfig}
\usepackage{pstricks}
\usepackage{pst-plot}
\usepackage{pst-node}
\usepackage{graphicx}
\usepackage[english]{babel}

\theoremstyle{plain}
\newtheorem{thm}{Theorem}
\newtheorem{prp}[thm]{Proposition}
\newtheorem{lem}[thm]{Lemma}

\newtheorem{definition}{Definition}
\newtheorem{rem}{Remark}

\newcommand{\R}{\mathbb{R}}

\newcommand{\be}{\begin{equation}}
\newcommand{\ee}{\end{equation}}
\newcommand{\beq}{\begin{eqnarray}}
\newcommand{\eeq}{\end{eqnarray}}
\newcommand{\ffi}{\varphi}
\newcommand{\pat}{\partial}
\newcommand{\p}{\parallel}
\newcommand{\ep}{\epsilon}
\newcommand{\epa}{\epsilon_{1}}
\newcommand{\un}{u_{\nu}}
\newcommand{\xn}{\xi_{\nu}}
\newcommand{\zn}{z_{\nu}}

\newcommand{\betan}{\beta_{\nu}}

\begin{document}
\title{Stochastic incompressible Euler equations in a two-dimensional domain}
\date{}
\author{Hakima Bessaih\footnote{University of Wyoming, Department of Mathematics, Dept. 3036, 1000
East University Avenue, Laramie WY 82071, United States, bessaih@uwyo.edu} }
 
\maketitle

\begin{abstract}
The aim of these notes is to give an overview of the current results about existence and uniqueness of solutions for the stochastic Euler equation driven by a Brownian noise in a two-dimensional bounded domain.
\end{abstract}

\section{Introduction}
We are concerned with the following system of Euler equations 
\be\left\{\begin{array}{l}
           \frac{\pat u}{\pat t} + u\cdot\nabla u + \nabla p =
           f + G(u) \frac{\pat W}{\pat t},\\
           \nabla\cdot u = 0,\\
           u\cdot n|_{\pat D} = 0,\\
           u(x,0) = u_{0}.
           \end{array}
   \right.\label{SEE}\ee
Here $D$ is a regular open domain of $\mathbb{R}^{2}$ with boundary $\pat D$, $n$ is
the exterior normal to $\pat D$, $u=(u_{1},u_{2})$ is the velocity field of
the fluid, $p$ is the pressure field, $u_{0}$ is the initial velocity field,
$f$ is the body force field, $W$ is a Brownian motion and $G$ is an 
operator acting on the solution. 

The deterministic case (when $G(u)=0$) has been extensively studied and we refer to the book of P. L. Lions \cite{PLions96} for a concise and complete analysis. Moreover, the book has a very nice introduction to the fundamental equations for newtonian fluids for which the Euler equation is an approximation. The above boundary condition is called the slip boundary condition,
and  is very different from the no-slip condition satisfied by the classical Navier-Stokes equations.

A few papers have been dealing with the two-dimensional stochastic Euler equation. 
Among them \cite{BF99} and \cite{K02} for pathwise global weak solutions, \cite{B99} for martingale global solutions in Hilbert spaces and \cite{BP01} for martingale global solution in some Banach spaces of $L^{q}$-type and \cite{GV} for some smooth solutions. For the whole space $\R^{d}$, smooth local solutions have been studied in 
\cite{MV} for $d=2, 3$ and \cite{K09} for $d=3$. Let us mention that in the three-dimensional case only local solutions can be proven, see \cite{MV, K09, GV}. The longtime behavior of a modified version of these equations have been studied in \cite{B00, B08} through their stochastic attractors and stationary solutions. In particular, a linear dissipation was added. We don't know any result concerning invariant measures or stationary solutions for the system \eqref{SEE} without adding a dissipation. This problem 
seems to be related to the topic  of turbulence theory and the dissipation of energy and/or enstrophy.

These notes are devoted to the study of the Euler equations in a bounded two-dimensional domain. 
The study is focused on the pathwise and martingale global weak solutions, that is, solutions with values in $W^{1,2}$ or  $W^{1,q}$-type.  

In Section 3 the stochastic Euler equation is  driven by an additive noise. We are able to prove existence and uniqueness of solutions. The approach that we have chosen is a little bit long but relies only on elementary facts and is essentially self-contained. The uniqueness is based on a method of Yudovich
(which requires proper modifications in the stochastic case) revisited by \cite{bardos} and \cite{kato}.  

Section 4 is devoted to system \eqref{SEE} with a fully multiplicative noise. In this section, we chose a different approach  with respect to previous one. This is an approach that uses deep regularity properties of a Stokes operator associated to equation \eqref{SEE}. We prove the existence of global $W^{1,2}$-valued martingale solutions by means of a compactness argument. In this case, the uniqueness is an open problem.

Section 5 is devoted to system \eqref{SEE} with a multiplicative noise in some Banach spaces. 
The existence of global  $W^{1,q}, (q\geq 2)$-valued martingale solutions are proved by means a compactness argument. This  result is due to \cite{BP01}.

\section{Preliminaries and notations.}
Let ${\cal V}$ be the space of infinitely differentiable vector field $u$
on $D$ with compact support strictly contained in $D$, satisfying
$\nabla\cdot u=0$. We introduce the space $H$ of all measurable vector fields 
$u: D\longrightarrow \R^{2}$ which are square integrable, divergence free, 
and tangent to the boundary

$$H = \left\{u\in\left[L^{2}(D)\right]^{2};\ \nabla\cdot u=0\ {\rm in}\
D,\ u\cdot n=0\ {\rm on}\ \pat D\right\};$$
the meaning of the condition $u\cdot n =0\ {\rm on}\ \pat D$ for such vector
fields is explained for instance in \cite{temam84}. The space $H$ is a separable
Hilbert space with inner product of $\left[L^{2}(D)\right]^{2}$, denoted in
the sequel by $<.,.>$ (norm $|.|$). Let $V$ be the following subspace of $H$;

$$V = \left\{u\in\left[H^{1}(D)\right]^{2};\ \nabla\cdot u=0\ {\rm in}\
D,\ u\cdot n=0\ {\rm on}\ \pat D\right\};$$
The space $V$ is a separable Hilbert space with inner product of
$\left[H^{1}(D)\right]^{2}$ (norm $\p.\p$). Identifying $H$ with its dual
space $H'$, and $H'$ with the corresponding natural subspace of the dual
space $V'$, we have the standard triple $V\subset H\subset V'$ with continuous
dense injections. We denote the dual pairing between $V$ and $V'$ by the
inner product of $H$.

Let $H$ be a separable Hilbert space. Given $p>1$ and $\gamma\in(0,1)$, 
let $W^{\gamma,2}(0,T;H)$ be the Sobolev space of all $u\in L^{p}(0,T;H)$
such that
$$\int_{0}^{T}\int_{0}^{T}\frac{|u(t)-u(s)|^{p}}{|t-s|^{1+\gamma p}}dtds
<\infty$$
endowed with the norm
$$\p u\p^{p}_{W^{\gamma,2}(0,T;H)} = \int_{0}^{T}|u(t)|^{p}dt
+ \int_{0}^{T}\int_{0}^{T}\frac{|u(t)-u(s)|^{p}}{|t-s|^{1+\gamma p}}dtds.$$
For $q\in (1,\infty)$, let us set $H^{1,q}:=H \cap W^{1,q}$. Let $K$ another Hilbert space. 
Let us denote by  $L_{2}(K,H)$ the set of 
Hilbert-Schmidt operators from $K$ to $H$ and by $R(K, W^{1,q})$ the space of all
$\gamma$-radonifying mappings from $K$ into $W^{1,q}$, we refer to \cite{BP01} for a detailed introduction to these spaces.

In the sequel,  for a vector field $u=(u_{1},u_{2})$  and $(x_{1},x_{2})\in D$, we will denote the rotational of $u$ by
$$\nabla\wedge u: = \frac{\pat u_{2}}{\pat x_{1}}
-\frac{\pat u_{1}}{\pat x_{2}}.$$
\section{The stochastic Euler equation with additive noise}
The results of this section can be found in \cite{BF99}. 
Similar results using different techniques can also be found in \cite{K02}.

\subsection{Functional setting, assumptions and main results}
\be\left\{\begin{array}{lr}
                \frac{\pat u}{\pat t} + (u\cdot\nabla)u + \nabla p =
                 f + \frac{\pat W}{\pat t},&{\rm in}\ (0,T)\times D\\
                 \nabla\cdot u=0, &{\rm in}\ (0,T)\times D\\
                 u\cdot n = 0, &{\rm on}\ (0,T)\times \pat D\\
                 u|_{t=0}=u_{0}.&{\rm in}\ D
           \end{array}
     \right. \label{1}\ee

Here, $W=W(t,\omega), t\geq 0, \omega\in \Omega$ is  an $H$-valued stochastic process on the probability space $(\Omega, \mathcal{F}, P)$ (for instance a Wiener process), subject to the following regularity in space: for $P$-a..e. $\omega\in\Omega$,   
\be W\in C\left([0,T];\left[H^{4}(D)\right]^{2}\cap V\right)\label{2}\ee
with the mapping $\omega\rightarrow W(\cdot,\omega)$ measurable in this topology, and
\be \nabla\wedge W = 0\ \ {\rm on}\ \ (0,T)\times\pat D.\label{3}\ee
To simplify the expression (even if some partial results require less 
assumptions), we impose throughout the paper the following conditions on 
$f$ and $u_{0}$

\be u_{0}\in V,\ \ f\in L^{2}(0,T; V).\label{4}\ee

Here are the main results of this section:

\begin{thm}\label{t1}
Under the assumptions \eqref{2}, \eqref{3} and \eqref{4}, there exists 
(at least) an adapted process $u(t,\omega)$ solution of 
\eqref{1} in the sense that for $P$-a.e. $\omega\in\Omega$
$$u(.,\omega)\in C(0,T;H)\cap L^{2}([0,T]; V)$$
and 
\beq
<u(t),\phi> &+& \int_{0}^{t}<(u(s)\cdot\nabla)u(s),\phi>ds =
<u_{0},\phi>\nonumber\\
&+& \int_{0}^{t}<f(s),\phi>ds + <W(t),\phi>,\label{5}
\eeq
for every $t\in [0,T]$ and every $\phi\in V$
\end{thm}

\begin{thm}\label{t2}
If in addition $\nabla\wedge u_{0}\in L^{\infty}(D)$, $\nabla\wedge f\in L^{\infty}([0,T]\times D)$
and $(\Delta\nabla\wedge W)\in L^{\infty}([0,T]\times D)$, the solution of problem \eqref{1} is unique.
\end{thm}

\subsection{Approximating Navier-Stokes equations.}
The main effort to prove the previous theorems consists in the analysis of
the following equation. For every $\nu>0$ we consider the equation of 
Navier-Stokes type
\be\left\{\begin{array}{lr}
                \frac{\pat u}{\pat t} + (u\cdot\nabla) u + \nabla p =
                 \nu\Delta u + f + \frac{\pat W}{\pat t},
                 &{\rm in}\ (0,T)\times D\\
                 \nabla\cdot u=0, &{\rm in}\ (0,T)\times D\\
                 \nabla\wedge u=0, &{\rm on}\ (0,T)\times \pat D\\
                 u\cdot n = 0, &{\rm on}\ (0,T)\times \pat D\\
                 u|_{t=0}=u_{0},&{\rm in}\ D
           \end{array}
     \right. \label{6}\ee
where 
$$\nabla\wedge u=\frac{\pat u_{2}}{\pat x_{1}}
-\frac{\pat u_{1}}{\pat x_{2}}$$
is the vorticity. Due to the boundary condition $\nabla\wedge u=0$, this is 
not the classical equation for a viscous fluid in a boundary domain, but it 
can be studied in a similar way as we shall show below.

Let us consider in each point $\sigma_{0}\in\pat D$ the reference frame 
composed by the 
exterior normal $n$ and the tangent $\tau$. Let $(x_{\tau},x_{N})$ the 
components of points of $R^{2}$ and $(u_{\tau},u_{N})$ the components of 
a vector with respect to this reference.
If $(n_{\tau},n_{N})$ are the components of the exterior normal $n$, the 
curvature of $\pat D$ at the point $\sigma_{0}$ is given by the relation

\be k(\sigma_{0}) = -\frac{\pat n_{\tau}}{\pat x_{\tau}}.\label{7}\ee
On the other hand, the rotational of a field $u$ will be written in the 
reference$(\tau,n)$

\be \nabla\wedge u=\frac{\pat u_{N}}{\pat x_{\tau}}
-\frac{\pat u_{\tau}}{\pat x_{N}}.\label{8}\ee
The relation 

$$ a(u,v)=\int_{D}\nabla u\cdot\nabla v 
-\int_{\pat D}k(\sigma)u(\sigma)\cdot v(\sigma)d\sigma,$$
defines a continuous and coercive bilinear form on $V$ and
$$b(u,v,w)=\int_{D}(u\cdot\nabla)v\cdot w,$$
defines a continuous trilinear form on $V$.
The following properties on $a$ and $b$ will be used

For arbitrary $\epsilon>0$ (see \cite{Lions61})
$$\int_{\pat D}k(\sigma)|u(\sigma)|^{2}d\sigma\leq\epsilon|\nabla u|^{2} 
+ C(\epsilon)|u|^{2}.$$
and for all $u, v, w\in V$ 
$$b(u,v,w)=-b(u,w,v)\quad {\rm and}\quad b(u,v,v)=0.$$

\begin{definition}
We say that a stochastic process $u(t,\omega)$ is a weak solution for \eqref{6} if
$$u(.,\omega)\in C([0,T]; H)\cap L^{2}(0,T; V)$$
for $P$-a.e. $\omega\in\Omega$,
such that for every $t\in [0,T]$ and every $\phi\in V$, it satisfies
$P$-a.e. the equation
\beq
& &<u(t),\phi> + \int_{0}^{t}a(u(s),\phi)ds + 
\int_{0}^{t}b(u(s),u(s),\phi)ds =\nonumber\\
& &<u_{0},\phi> + \int_{0}^{t}<f(s),\phi>ds + <W(t),\phi>.
\label{9}\eeq
\end{definition}

\begin{prp}\label{u}
There exists a unique weak solution $u(.,\omega)$ of the previous 
Navier-Stokes equations. Moreover, for every function
$$\ffi(.)\in C([0,T]; H)\cap L^{2}(0,T; V)\cap H^{1}(0,T; V')$$
such that
$$\ffi(T)=0,$$
we have
\beq
& &\int_{0}^{T}\left<u(s),\frac{\pat\ffi}{\pat s}(s) 
+ (u(s)\cdot\nabla)\ffi(s)\right>ds - \int_{0}^{T}a(u(s),\ffi(s))ds 
=\nonumber\\
&- &<u_{0},\ffi(0)> - \int_{0}^{T}<f(s),\ffi(s)>ds 
+ \int_{0}^{T}\left<W(s),\frac{\pat\ffi}{\pat s}(s)\right>.
\label{10}\eeq
\end{prp}
\begin{proof}
The proof of the proposition above is based on a pathwise argument through the change of variable
$z=u-W$. Fix $\omega\in\Omega$, then the equation \eqref{9} becomes 
\be \left\{\begin{array}{l}
                 <\dot{z},\phi> + {a}(z,\phi) + b(z+W,z+W,\phi)=
                 <f,\phi> - {a}(W,\phi),\vspace{3mm}\\
                 z(0) = u_{0}.
           \end{array}
     \right.\label{11}\ee 
In particular, the proof will be based on the following steps:
\begin{enumerate}
\item An approximate solution $z_{m}$ of \eqref{11} will be constructed for example through  Faedo-Galerkin method. In particular, $z_{m}$ is uniquely determined. 
\item Some a priori estimates will be performed on $z_{m}$ in 
$L^{2}([0,T]; V)\bigcap  W^{1,2}([0,T]; V')$.
\item The passage to the limit on $m$ will be performed using classical compactness arguments, see the appendix.

\item All the statements of the previous steps have been made for any given $\omega$. In this way, we know that the solution $u$ corresponding to a given $\omega$ is the limit (in the appropriate sense) of the entire sequence of Galerkin approximation (thus we don't need to work with subsequences depending on $\omega$). Since these approximations are measurable in $\omega$, the limit function 
$\omega\rightarrow u$ is also measurable.

\item The adaptedness of $u$ is being obtained by a limiting procedure over successive approximations 
that are adapted at each step (in particular $z_{m}$ is an adapted process). 
\end{enumerate}

\end{proof}

\begin{prp}\label{beta}
Let $\beta = \nabla\wedge u$, where $u$ is the solution of the previous 
Navier-Stokes equations (hence in particular $\beta\in L^{2}((0,T)\times D)$).
For every function 
$$\psi(.)\in C([0,T]; H_{0}^{1}(D))\cap L^{2}(0,T; H^{2}(D))
\cap H^{1}(0,T; L^{2}(D))$$
such that
$$\psi(T)=0,$$
we have
\beq
& &\int_{0}^{T}\left<\beta(s),\frac{\pat\psi}{\pat s}(s)
+ \Delta\psi(s) + (u(s)\cdot\nabla)\psi(s)\right>ds 
= -<\nabla\wedge u_{0},\psi(0)> \nonumber\\
& & - \int_{0}^{T}<\nabla\wedge f(s),\psi>ds
+ \int_{0}^{T}\left<\nabla\wedge W(s),\frac{\pat\psi}{\pat s}(s)\right>.
\label{16}
\eeq
\end{prp}
\begin{proof}
We plug in particular $\ffi = \nabla^{\bot}\psi$, where 
$\nabla^{\bot} = (D_{2},-D_{1})$ in \eqref{10}; we obtain 

\beq
& &\int_{0}^{T}<u,\frac{\pat\nabla^{\bot}\psi}{\pat s}> 
+ \int_{0}^{T}<u,u\cdot\nabla\nabla^{\bot}\psi> 
- \int_{0}^{T}a(u,\nabla^{\bot}\psi) = \nonumber\\
& & -<u_{0},\nabla^{\bot}\psi_{0}>
- \int_{0}^{T}<f,\nabla^{\bot}\psi> 
+ \int_{0}^{T}<W,\frac{\pat\nabla^{\bot}\psi}{\pat s}>.
\label{17}\eeq
Using the fact $\psi|_{\pat D}=0$ and the integration by part for the first 
term on the left hand side of the above inequality we obtain

$$<u,\frac{\pat\nabla^{\bot}\psi}{\pat s}> = 
-<\nabla\wedge u,\frac{\pat\psi}{\pat s}>.$$
We apply an integration by part for the second term on the left hand 
side of \eqref{17}; we have
\beq
<u,u\cdot\nabla\nabla^{\bot}\psi> &=& 
\int_{D}u_{i}u_{j}D_{j}(\nabla^{\bot}\psi)_{i}\nonumber\\
&=&  \int_{D}u_{i}u_{j}D_{j}D_{i}^{\bot}\psi\nonumber\\
&=& - \int_{D}D_{i}^{\bot}(u_{i}u_{j})D_{j}\psi 
+ \int_{\pat D}u_{i}u_{j}D_{j}\psi n_{i}^{\bot}\nonumber\\
&=& -\int_{D}(D_{i}^{\bot}u_{i})u_{j}D_{j}\psi
- \int_{D}u_{i}(D_{j}^{\bot}u_{j}D_{j}\psi\nonumber\\
&+&\int_{\pat D}u_{i}u_{j})D_{j}\psi n_{i}^{\bot}
\label{18}\eeq
Another integration by part for the second term on the right hand side 
of \eqref{18} yields that
                       
\beq 
\int_{D}u_{i}(D_{j}^{\bot}u_{j})D_{j}\psi 
&=&- \int_{D}D_{j}(u_{i}(D_{i}^{\bot}u_{j}))\psi 
+ \int_{\pat D}u_{i}(D_{i}^{\bot}u_{j})\psi n_{j}\nonumber\\
&=& - \int_{D}D_{j}u_{i}(D_{i}^{\bot}u_{j})\psi
 - \int_{D}u_{i}(D_{j}D_{i}^{\bot}u_{j})\psi\nonumber\\
&+& \int_{\pat D}u_{i}(D_{i}^{\bot}u_{j})\psi n_{j}.
\label{19}
\eeq
Because of 
$$\nabla.u = 0,$$
the first and the second term on the right hand side of \eqref{19} are equal to
zero and we obtain
$$\int_{D}u_{i}(D_{j}^{\bot}u_{j})D_{j}\psi = 
\int_{\pat D}u_{i}(D_{i}^{\bot}u_{j})\psi n_{j}.$$
So that

$$<u,u\cdot\nabla\nabla^{\bot}\psi> = - <\beta,u\cdot\nabla\psi> 
- \int_{\pat D}\left(u_{i}(D_{i}^{\bot}u_{j})\psi n_{j}
- u_{i}u_{j}D_{j}D_{i}^{\bot}\psi n_{i}^{\bot}\right).$$
Since we have the following hypothesis on the boudary $\pat D$
$$u\cdot n|_{\pat D} = 0,\ \  {\rm and}\ \ 
\nabla^{\bot}\psi\cdot n|_{\pat D}=0,$$
we have
$$ \int_{\pat D}\left(u_{i}(D_{i}^{\bot}u_{j})\psi n_{j}
- u_{i}u_{j}D_{j}D_{i}^{\bot}\psi n_{i}^{\bot}\right)
= \int_{\pat D}u\cdot n\left(u_{1}D_{2}\psi-D_{1}\psi u_{2}\right) = 0,$$
which implies that 

$$<u,u\cdot\nabla\nabla^{\bot}\psi> = - <\beta,u\cdot\nabla\psi>.$$
We apply the integration by part twice for the third term on the left 
hand side of \eqref{17}, we get

\beq
a(u,\nabla^{\bot}\psi) &=& \int_{D}D_{i}u_{j}D_{i}D_{j}^{\bot}\psi  
- \int_{\pat D}ku\cdot\nabla^{\bot}\psi\nonumber\\
&=&-\int_{D}u_{j}D_{i}^{2}D_{j}^{\bot}\psi
+ \int_{\pat D}\left(u_{j}D_{i}D_{j}^{\bot}\psi n_{i} - 
ku\cdot\nabla^{\bot}\psi\right)\nonumber\\
&=&\int_{D}D_{j}^{\bot}u_{j}D_{i}^{2}\psi\nonumber\\
& &+ \int_{\pat D}\left(u_{j}D_{i}D_{j}^{\bot}\psi n_{i} 
- u_{j}D_{i}^{2}\psi n_{j}^{\bot}
- ku_{j}D_{j}^{\bot}\psi\right).
\label{20}
\eeq
Since $u\cdot n=0\ \ {\rm on}\ \ \pat D$, we have:
$$\frac{\pat}{\pat x_{\tau}}(n_{\tau}u_{\tau}+n_{N}u_{N})=0.$$
We have also
$$(D_{N}\psi)n_{\tau}-(D_{\tau}\psi)n_{N}=0.$$
Using \eqref{7} and the above boundary conditions, the last boundary 
integral in \eqref{20} is equal to zero.\\
It remains to apply the integration by part for the integrals in the right 
hand side of \eqref{17} to have the result.
\end{proof}

\begin{rem}
When a function $\beta\in L^{2}(0,T)\times D)$ satisfies the previous 
variational equation given in Proposition \ref{beta}, then we call it a {\it generalized} {\it solution} of the 
following equations
\be\left\{\begin{array}{lr}
                \frac{\pat\beta}{\pat t} + (u\cdot\nabla)\beta 
                 = \nu\Delta\beta + \nabla\wedge f 
                 + \frac{\pat\nabla\wedge W}{\pat t},
                 &{\rm in}\ (0,T)\times D\\
                 \beta = 0, &{\rm on}\ (0,T)\times \pat D\\
                 \beta|_{t=0}=\nabla\wedge u_{0},&{\rm in}\ D
           \end{array}
     \right. \label{21}\ee
\end{rem}
In particular in \eqref{21}, $\beta=\nabla\wedge u$. Now, let us state  the uniqueness of generalized solutions of  \eqref{21}

\begin{prp}\label{beta-uniqueness}
There exists a unique generalized solution 
$\beta\in L^{2}\left((0,T)\times D\right)$ 
of the previous equation \eqref{21}.
\end{prp}
\begin{proof}
Assume that $\beta'$ and $\beta''$ are generalized solutions, and 
set $\beta=\beta'-\beta''$. Then $\beta$ is a generalized solution
with data equal to zero, i.e. it satisfies

$$\int_{0}^{T}\left<\beta(s),\frac{\pat\psi}{\pat s}(s)
+ \nu\Delta\psi(s) + (u(s)\cdot\nabla)\psi(s)\right>ds = 0,$$
for all $\psi(.)$ as in Proposition \ref{beta}. Now using Lemma \ref{parab}, 
when $\psi(.)$ varies, the expression 
$\frac{\pat\psi}{\pat s}(s) + \nu\Delta\psi(s) + (u(s)\cdot\nabla)\psi(s)$
describes a dense set in $L^{2}((0,T)\times D)$ (in fact the whole space).
Hence,
$$\int_{0}^{T}\left<\beta(s),\widetilde{\psi}(s)\right>ds = 0,$$
for an arbitrary $\widetilde{\psi}\in L^{2}((0,T)\times D)$ which implies that
$\beta = 0$.
\end{proof}

\begin{prp}
The solution $\beta$ given by Proposition \ref{beta} satisfies
$$\beta\in C\left([0,T];L^{2}(D)\right)\cap 
L^{2}\left(0,T;H_{0}^{1}(D)\right).$$
Moreover, it satisfies
\be\p\beta\p_{C([0,T];L^{2}(D))}\leq C,\label{22}\ee
where the constant is independent of $\nu$ (it depends on $W$ and 
$u_{0}$).
\end{prp}
\begin{proof} 
We prove that there exists a solution $\beta$ of equation \eqref{21} with 
such regularity. Since it is automatically a generalized solution (as it can 
be verified as Proposition \ref{beta}), it coincides with the generalized solution 
given by Proposition \ref{beta-uniqueness}.\\
Setting formally $z=\beta -\nabla\wedge W$ we get the equation
\be\left\{\begin{array}{lr}
                \frac{\pat z}{\pat t} + (u\cdot\nabla)z
                 = \nu\Delta z + g, &{\rm in}\ (0,T)\times D\\
                 z = 0, &{\rm on}\ (0,T)\times \pat D\\
                 z|_{t=0} = v_{0},&{\rm in}\ D
           \end{array} \right. \label{23}\ee
where
$$g=\nabla\wedge f-(u\cdot\nabla)(\nabla\wedge W)
+\nu\Delta(\nabla\wedge W).$$
Since $\nabla\wedge u_{0}\in L^{2}(D)$ and $g\in L^{2}(0,T;H^{-1}(D))$,
there exists a unique solution
$$z\in C\left([0,T];L^{2}(D)\right)\cap
L^{2}\left(0,T;H_{0}^{1}(D)\right)\cap H^{1}\left(0,T;H_{0}^{-1}(D)\right).$$
By Lemma \ref{parab}, it is strightforward to see that 
$\beta:=z+\nabla\wedge W$
is a solution of the first equation of \eqref{38} and has the regularity 
required by the previous proposition. Moreover, since 
$\int_{D}((u\cdot\nabla)z)z = 0$
\beq
& &\frac{1}{2}\frac{d}{dt}|z|_{L^{2}(D)}^{2} =
\int_{D}(\nu\Delta z + g - (u\cdot\nabla)z)\Delta z
\nonumber\\
& &= -\nu|z|_{H^{1}(D)}^{2} + C|z|_{L^{2}(D)}|g|_{L^{2}(D)}.\nonumber
\eeq
Hence for all $t\in [0,T]$,
\beq
|z(t)|_{L^{2}(D)}^{2}&\leq& |\nabla\wedge u_{0}|_{L^{2}(D)}^{2} 
+ C(\int_{0}^{T}|z(s)|_{L^{2}(D)}^{2}ds
+\int_{0}^{T}(|\nabla\wedge f(s)|_{L^{2}(D)}^{2}\nonumber\\ 
&+&|\nabla\nabla\wedge W(s)|_{L^{\infty}(D)}|u(s)|_{L^{2}(D)}
+|\Delta\nabla\wedge W(s)|_{L^{2}(D)}^{2})ds).\nonumber
\eeq
Using Gronwall's Lemma and some previous estimates completes the proof.
\end{proof}

\begin{prp}\label{u2}
The solution $u$ given by proposition \ref{u} satisfies   
$$u\in C\left([0,T];H^{1}(D)\right)\cap
L^{2}\left(0,T;H^{2}(D)\right).$$
Moreover, it satisfies
$$\p u\p_{C([0,T];H^{1}(D))}\leq C$$
where the constant is independent of $\nu$.
\end{prp}

\begin{proof}
Observe that since $\nabla\cdot u=0$, $u$ satisfies the following 
elliptic system

\be\left\{\begin{array}{l}
                \Delta u = - \nabla^{\bot}\beta,\\
                \beta|_{\pat D} = 0,\\
                u\cdot n|_{\pat D} = 0.
          \end{array}
    \right.\label{24}\ee
   
We multiply the first equation of \eqref{24} by $u$ and integrate over $D$.
Since $\beta|_{\pat D}=0$, we obtain

\be|\nabla u|_{L^{2}(D)}^{2} = \int_{\pat D}\nabla u\cdot u\cdot n 
+ <\beta,\nabla\wedge u>.\label{25}\ee 
For an arbitrary $\epsilon >0$ we have (see \cite{Lions61})
$$\int_{\pat D}\nabla u\cdot u\cdot n\leq\epsilon |\nabla u|_{L^{2}(D)}^{2}
+C(\epsilon)|u|_{L^{2}(D)}^{2},$$ 
Which yields that 
\be|\nabla u|_{L^{2}(D)}^{2}\leq C\left(|\beta|_{L^{2}(D)}^{2} 
+ |u|_{L^{2}(D)}^{2}\right).\label{26}\ee
The result follows from Proposition \ref{u}, \eqref{22} and  \eqref{26}.
\end{proof}

\subsection{Proof of the main results}

We are now able to proof the main results of the previous section

\subsubsection{Proof of Theorem \ref{t1} (Existence)}

Let us emphasize again that the path of the Brownian motion $W$ is given and that our arguments 
are $\omega$-wise 
(hence purely deterministic). In particular, our statements about uniform boundedness have to be understood $\omega$-wise.

From Proposition \ref{u2}, we have that $\un$ weak solution of Navier Stokes 
equations  \eqref{6} is uniformly bounded in
$L^{\infty}(0,T; [H^{1}(D)]^{2})$ so in $L^{\infty}(0,T; V)$. This implies 
in particular that 
\be \lim_{\nu\rightarrow 0}a(\un,\phi)=0,\ \ \ \forall\phi\in V.\label{27}\ee
Since $\un$ satisfies  \eqref{9}, $\un$  remains bounded in
$L^{2}(0,T; V')$. The embedding $V\subset H$ being compact, this implies 
that we can extract from $\un$ a subsequence (also called $\un$) which 
converge weakly in $L^{2}(0,T; V)$  and strongly in $\left(L^{2}([0,T]\times D)\right)^{2}$
(use Lemma \ref{compact1}). On the other hand $\un'$ converge weakly in
$L^{2}(0,T; V)$; we deduce that $u\in C([0,T]; H)$  (see \cite{temam84}) and verifies that
\be u(0)=u_{0}.\label{28}\ee
From strong convergence of $\un$ in $L^{2}([0,T]\times D)$ we deduce that
\be\lim_{\nu\rightarrow 0}<(\un\cdot\nabla)\un,\phi>= 
<(u\cdot\nabla)u,\phi>\ \ \ \ \forall\phi\in V.\label{29}\ee
It follows from  \eqref{27}, \eqref{28} and  \eqref{29} that the limit $u$ 
is solution of  \eqref{5}.

The measurability of $u$ follows from using Lemma \ref{mes} with the choice of
$$X=\left\{x\in C\left([0,T];\left[H^{4}(D)\right]^{2}\cap V\right),\quad \nabla\wedge x=0 \quad on 
(0,T)\times\partial D\right\}$$
and 
$$Y=C([0,T]; H)\cap L^{2}([0,T]; V).$$
Finally, the adaptedness of the process $u$ follows from a limiting procedure of adapted processes.
  
\subsubsection{Proof of Theorem \ref{t2} (Uniqueness)}
Let us set $Q=[0,T]\times D$. Then, we have the following Lemma
\begin{lem}\label{up}
Under the assumptions of Theorem \ref{t2}, 
$\un$ and its limit $u$ are in 
$X_{p}= L^{\infty}(0,T; (W^{1,p}(D))^{2})$ ($1\leq p<\infty$). Besides 
$\un$ and $u$ satisfy the following inequality
\beq 
|u|_{X_{p}}&\leq& Cp(|\nabla\wedge u_{0}|_{L^{\infty}(D)}
+ |u_{0}|_{L^{\infty}(0,T; V)} + |\nabla\wedge f|_{L^{\infty}(Q)}\nonumber\\ 
&+& |\Delta\nabla\wedge W|_{L^{\infty}(Q)}
+ |f|_{L^{\infty}(0,T;V)}).\label{30}
\eeq
\end{lem}
\begin{proof}
$\un\in L^{\infty}(0,T; V)$ (Proposition \ref{u2}). Using an embedding theorem, 
$\un$ remains bounded in $L^{\infty}(0,T; [L^{4}(D)]^{2})$.
On the other hand $\zn=\beta_{\nu}-\nabla\wedge W$ is solution of \eqref{23}, 
where $\beta_{\nu}=\nabla\wedge\un$. We multiply the first equation
of \eqref{23} by $|\zn|^{2}\zn$ and integrate over $D$ we obtain
\beq
\frac{1}{4}\frac{d}{dt}\int_{D}|\zn|^{4} &+& 
\nu\int_{D}(\nabla\zn)^{2}|\zn|^{2}
=\int_{D}(\nabla\wedge f)|\zn|^{2}\zn 
+\nu\int_{D}(\Delta\nabla\wedge W)|\zn|^{2}\zn\nonumber\\ 
&+& \int_{D}(\un\cdot\nabla)\zn|\zn|^{2}\zn
- \int_{D}(\un\cdot\nabla)(\nabla\wedge W)|\zn|^{2}\zn\label{31}
\eeq
Using H\"{o}lder's inequality for the terms in the right hand side of 
\eqref{31} and then the Gronwall's Lemma we obtain that 
\beq
& &\sup_{0<t<T}|\zn(t)|^{4}_{L^{4}(D)}\leq (|z_{0}|^{4}_{L^{4}(D)}\nonumber\\
&+& \int_{0}^{T}\left(|\nabla\wedge f|^{4}_{L^{4}(D)}
+ |\Delta\nabla\wedge W|^{4}_{L^{4}(D)}
+ |\nabla\nabla\wedge W|^{4}_{L^{\infty}(D)}|\un|^{4}_{L^{4}(D)}\right))
{\rm e}^{T}.\nonumber
\eeq

Under the hypothesis of the Lemma, it yields that 
$\betan\in L^{\infty}(0,T; L^{4}(D))$.      
Using the system \eqref{24} which is elliptic, we deduce that 
$\un\in L^{\infty}(0,T; W^{1,4}(D))$ which implies by Sobolev theorem
that $\un\in L^{\infty}(Q)$. 
Now we apply the maximum principle for $\zn$ solution of \eqref{23}; we get

\beq 
|\zn|_{L^{\infty}(Q)}&\leq& C(|\nabla\wedge u_{0}|_{L^{\infty}(D)} 
+ |\nabla\wedge f|_{L^{\infty}(Q)} \nonumber\\
& &+ |\Delta\nabla\wedge W|_{L^{\infty}(Q)}
+ |(\un\cdot\nabla)\nabla\wedge W|_{L^{\infty}(Q)}).\label{32}
\eeq 
As in \cite{bardos}, using the system \eqref{24} and in virtue of \eqref{32},
we obtain the following estimate
\beq
|\un|_{X_{p}}&\leq& Cp(|\nabla\wedge u_{0}|_{L^{\infty}(D)}
+ |\nabla\wedge f|_{L^{\infty}(Q)} + |\Delta\nabla\wedge W|_{L^{\infty}(Q)}
\nonumber\\
& &+ |(\un\cdot\nabla)\nabla\wedge W|_{L^{\infty}(Q)} 
+ |\un|_{[L^{p}(Q)]^{2}}),\label{33}
\eeq
where $C$ is a constant independent of $\nu$.
According to the hypothesis of the Lemma and following the argument of \cite{bardos},
we obtain \eqref{30} for $\un$. Passing to the limit on $\nu$,
we obtain the same estimate for $u$.
\end{proof}

\begin{lem}\label{kato}
We have
\be |v|_{L^{p}(D)}\leq Cp^{1/2}|v|_{H^{1}(D)},\ \ {\rm for}\ 2\leq p <\infty,
v\in H^{1}(D).\label{34}\ee
\end{lem}
{\bf Proof}: See \cite{kato} 

Let us assume that $u'$ and $u''$ are generalized solutions of \eqref{1} with the 
same initial data and the same external body force and set $u=u'-u''$, then 
$u$ is a generalized solution with data equal to zero i.e. it satisfies 
\be\left\{\begin{array}{lr}
          <u',\ffi> + <(u.\cdot\nabla)u,\ffi> = 0,&\forall\ffi\in V\\
           u(0) = 0.
          \end{array}
   \right.\label{35}\ee
In particular for $\ffi=u$, we have
\beq
\frac{d}{dt}|u|_{[L^{2}(D)]^{2}}^{2} &=& -2<(u.\cdot\nabla)u,u>\nonumber\\
&=&<(u.\cdot\nabla)u',\ffi>\nonumber\\
&=&-2\int_{D}(u.\cdot\nabla)u\cdot u.\nonumber
\eeq
By H\"{o}lder's inequality we have 
$$\frac{d}{dt}|u|_{[L^{2}(D)]^{2}}^{2}\leq 2|\nabla u'|_{[L^{p}(D)]^{2}}
|u|^{2}_{[L^{2p'}(D)]^{2}},\ \ 
\frac{1}{p}+\frac{1}{p'}=1.$$
Using the estimate \eqref{30}, we obtain
$$\frac{d}{dt}|u|_{[L^{2}(D)]^{2}}^{2}\leq Cp|u|^{2}_{[L^{2p'}(D)]^{2}}.$$
Now using an interpolation result we obtain
$$|u|_{[L^{2p'}(D)]^{2}}\leq |u|_{[L^{2}(D)]^{2}}^{1-\lambda}|u|^{\lambda}_{[L^{p}(D)]^{2}},\ 
\lambda=\frac{1}{p-2}\leq 1,\ 3\leq p<\infty.$$
Hence, we have
\beq
\frac{d}{dt}|u|_{[L^{2}(D)]^{2}}^{2\lambda}&=&
\lambda|u|_{[L^{2}(D)]^{2}}^{2\lambda-2}\frac{d}{dt}|u|_{[L^{2}(D)]^{2}}^{2}
\nonumber\\
&\leq&\lambda|u|_{[L^{2}(D)]^{2}}^{2\lambda-2}Cp|u|^{2}_{[L^{2p'}(D)]^{2}}
\nonumber\\
&\leq&\lambda Cp|u|_{[L^{p}(D)]^{2}}^{2\lambda}.\label{36}
\eeq
Using Lemma \ref{kato} and $\lambda=\frac{1}{p-2}$ and $u(0)=0$, then  
the integration of \eqref{36} leads to the estimate

\be |u(t)|_{[L^{2}(D)]^{2}}\leq 
(Ct)^{\frac{p-2}{2}}(\frac{p}{p-2})^{\frac{p-2}{2}}
(Cp)^{\frac{1}{2}}.\label{37}\ee
Suppose now that $t$ is so small that $Ct<1$. If we let $p\rightarrow\infty$,
then $(Ct)^{\frac{p-2}{2}}(Cp)^{\frac{1}{2}}\rightarrow 0$ while
$(\frac{p}{p-2})^{\frac{p-2}{2}}$ remains bounded. Hence $u(t)=0$ if $Ct<1$.
We repeat the argument in order to cover the whole interval $[0,T]$ so that the uniqueness is proved.
This completes the proof of Theorem \ref{t2}.

\section{The stochastic Euler equation with multiplicative noise}\label{weak}
All the results of this section can be found in \cite{B99}.

\subsection{Preliminaries}
Let $H, V$ the spaces previously defined.  Let $K$ be another separable 
Hilbert space and let $W(t)$ be a cylindrical Wiener process with values 
in $K$, defined on the stochastic basis $(\Omega, {\cal F}, \left\{{\cal F}_{t}\right\}_{t\geq 0}, P)$ 
(with the expectation $E$). 
Let $a(\cdot,\cdot)$ and $b(\cdot,\cdot,\cdot)$ the bilinear and trilinear forms previously defined.

Let us set 
$$D(A)=\left\{u\in V\cap (H^{2}(D))^{2}, 
\nabla\wedge u=0\right\},$$
and define the linear operator $A:D(A)\longrightarrow H$, as 
$<Au,v> = a(u,v)$.
We define the bilinear operator $B(u,v):V\times V\longrightarrow V'$,
as $<B(u,v),z> = b(u,v.z)$ for all $z\in V$. By the incompressibility
condition we have
$$<B(u,v),v> = 0,\ \ <B(u,v),z> = -<B(u,z),v>.$$
$B$ can be extended to a continuous operator
$$B: H\times H\longrightarrow D(A^{-\alpha})$$
for certain $\alpha > 1$.

In place of equations \eqref{SEE} we will consider the abstract stochastic 
evolution equation

\be\left\{\begin{array}{lr}
          du(t)+B(u(t),u(t))dt=f(t)dt+G(u)dW\\
          u(0)=u_{0},
          \end{array}
    \right.\label{abstract}\ee
for $t\in[0,T]$.

\subsection{Assumptions and main results}
Let assume that
$${\bf (H)} \quad u_{0}\in V \quad {\rm and} \quad f\in L^{\infty}([0,T], V),$$
and  

$$G(u)dW(x,t)=\sum_{i=1}^{\infty}C^{i}u(x,t)\frac{d\beta^{i}(t)}{dt},$$
where $\left\{\beta^{i}\right\}$ are independent Brownian motions,
$\left\{C^{i}\right\}\subset {\it L}(V,H)$ are linear operators 
satisfying for some positive real numbers $\lambda_{0}$, $\lambda_{1}$, 
$\lambda_{2}$, $C^{\infty}$ scalar fields 
$\left\{c^{i}\right\}$ and $\left\{b^{i}\right\}$ in $\bar{D}$

$${\bf (G1)}\left\{\begin{array}{l}
          C^{i}u(x,t) = c^{i}(x)u(x,t)+b^{i}(x),\\
          \sum_{i=1}^{\infty}|C^{i}u|^{2}\leq \lambda_{0}(|u|^{2}+1),\\
          \sum_{i=1}^{\infty}|\nabla\wedge(C^{i}u)|^{2}\leq
          \lambda_{1}|\nabla\wedge u|^{2} + \lambda_{2}(|u|^{2}+1).
          \end{array}
       \right.$$
For simplicity of computations, instead of {\bf (G1)}, we will consider the operator
$C^{i}u(x,t)=c^{i}(x)u(x,t)$ and the condition {\bf (G1)} becomes 
$${\bf (G1)'}\left\{\begin{array}{l}
          \sum_{i=1}^{\infty}|C^{i}u|^{2}\leq \lambda_{0}|u|^{2},\\
          \sum_{i=1}^{\infty}|\nabla\wedge(C^{i}u)|^{2}\leq
          \lambda_{1}|\nabla\wedge u|^{2} + \lambda_{2}|u|^{2}.
          \end{array}
       \right.$$
\begin{definition} Let $u_{0}\in V$. We say that there exists a martingale solution of the equation \eqref{abstract}
if there exists a stochastic basis
$(\Omega,{\cal F}, \left\{{\cal F}\right\}_{t\in [0,T]}, P)$,
a cylindrical Wiener process $W$ on the space $K$ and a progressively
measurable process $u:[0,T]\times\Omega\rightarrow H$, with $P$-a. s.
paths
$$u(.,\omega)\in C([0,T], D(A^{-\alpha}))\cap L^{\infty}(0,T; H)
\cap L^{2}(0,T; V)$$
such that $P$-a.s. the identity
\begin{eqnarray*}
<u(t),v> &+& \int_{0}^{t}<B(u(s),u(s)),v>ds = <u_{0},v> \\
&+& \int_{0}^{t}<f(s),v>ds + <\int_{0}^{t}G(u(s))dW(s),v>
\end{eqnarray*}
holds true for all $t\in[0,T]$ and all $v\in D(A^{\alpha})$.
\end{definition}
\begin{thm}\label{t3}
Under the assumption {\bf (H)} and {\bf (G1)}, there exists a martingale solution to the equation
\eqref{SEE}.
\end{thm}

\subsection{Proof of Theorem \ref{t3}}
The proof of Theorem \ref{t3} will be achieved following two steps. First, for a fixed $\nu>0$,
we introduce an approximating system (the modified Navier-Stokes system). Then, through some uniform estimates in $\nu$, we pass to the limit getting a weak solution in a probabilistic sense. This is achieved by means of the Prokhorov and Skorohod Theorems followed by a representation theorem for martingales.    

\subsubsection{Navier-Stokes equations and a priori estimates}
Let us consider for $\nu>0$ the system
\be\left\{\begin{array}{lr}
          du(t)+\nu Au+ B(u(t),u(t))dt=f(t)dt+G(u)dW\\
          u(0)=u_{0},
          \end{array}
    \right.\label{A}\ee
for $t\in[0,T]$.
Under the assumptions {\bf (H)} and {\bf (G1)}, system \eqref{A} has a strong solution 
$\un\in L^{2}(\Omega; C([0,T]; V))$ see \cite{CM}

\begin{lem}\label{up}
There exists a positive constant $C(p)$ independent of $\nu$ such that for each $p\geq 2$  
\be E(\sup_{0\leq s\leq t}|\un(s)|^{p}) \leq C(p),\label{up}\ee
\end{lem}
\begin{proof}
By It\^{o} formula, for $p\geq 2$ we have

\beq
d|\un(t)|^{p}\leq &p&|\un(t)|^{p-2}<\un,d\un>\nonumber\\
&+& (1/2)p(p-1)|\un(t)|^{p-2}\sum_{i=1}^{\infty}|C^{i}\un|^{2}dt.
\nonumber\eeq
Since $<B(\un,\un),\un>=0$ and using the hypothesis (G1)' we have

\beq
d|\un(t)|^{p}&+&\nu p|\un(t)|^{p-2}|\nabla\un|^{2}\leq
\nu p|\un(t)|^{p-2}(\int_{\pat D}k|\un|^{2})dt\nonumber\\
&+& p|\un(t)|^{p-2}<f,\un>dt
+ (1/2)\lambda_{0}p(p-1)|\un(t)|^{p}dt\nonumber\\ 
&+& p|\un(t)|^{p-2}\sum_{i=1}^{\infty}<C^{i}\un,\un>d\beta^{i}(t).\nonumber
\eeq
On the other hand, for an arbitrary $\ep>0$
$$\int_{\pat D}|\un|^{2}\leq \ep|\nabla\un|^{2} + C(\ep)|\un|^{2},$$
and by H\"{o}der inequality, for an arbitrary $\epa>0$
\beq
|\un(t)|^{p-2}<f,\un>&\leq&|\un(t)|^{p-2}|f||\un|\nonumber\\
&\leq&|\un(t)|^{p-2}(\frac{1}{2}|f|^{2}+\frac{1}{2}|\un|^{2})\nonumber\\
&\leq&\frac{1}{2}|\un(t)|^{p} + \frac{1}{2}|\un(t)|^{p-2}|f|^{2}\nonumber\\
\leq\frac{1}{2}\left(1+\epa(p-2)/p\right)|\un(t)|^{p} 
&+& \frac{1}{p\epa^{(p-2)/2}}|f|^{p}.\nonumber
\eeq
Thus,
\beq
d|\un(t)|^{p}&+&\nu p(1-\ep)|\un(t)|^{p-2}|\nabla\un|^{2}dt\leq
\frac{1}{p\epa^{\frac{(p-2)}{2}}}|f|^{p}\nonumber\\
&+&\left(\nu pC_{\ep}+\frac{p}{2}\left(1+\frac{\epa(p-2)}{p}\right)\right)
|\un(t)|^{p}dt\nonumber\\
&+& p|\un(t)|^{p-2}\sum_{i=1}^{\infty}<C^{i}\un,\un>d\beta^{i}(t).\nonumber
\eeq
Now we integrate between  $0$ and $t$ and take the supremum on $t$ and integrate 
over $\Omega$, we obtain
 
\begin{equation*}
\begin{split}
 & E(\sup_{0\leq s\leq t}|\un(t)|^{p})\leq
C\int_{0}^{t}E(\sup_{0\leq s\leq r}|\un(s)|^{p})dr\\
& + C \int_{0}^{t}E|f|^{p}ds
+ \sum_{i=1}^{\infty}pE\left(\sup_{0\leq s\leq t}\int_{0}^{s}
|\un(r)|^{p-2}<C^{i}\un,\un>d\beta^{i}(r)\right).
\end{split}
\end{equation*}

By Burkh\"{o}lder-Davis-Gundy inequality, for some constant $C$ we have
\beq
& &\sum_{i=1}^{\infty}E\left(\sup_{0\leq s\leq t}\int_{0}^{s}|\un(r)|^{p-2}
<C^{i}\un,\un>d\beta^{i}(r)\right)\leq\nonumber\\
& &CE\left(\left(\int_{0}^{t}
|\un(r)|^{2p-2}\sum_{i=1}^{\infty}|C^{i}\un|^{2}dr\right)^{1/2}\right)\leq
\nonumber\\
& &CE\left(\sup_{0\leq s\leq t}|\un(s)|^{p/2}\left(\int_{0}^{t}
\lambda_{0}|\un(r)|^{p}dr\right)^{1/2}\right).\nonumber
\eeq
By Gronwall Lemma we get \eqref{up}. 
\end{proof}

\begin{lem}
There exists a positive constant $C$ independent of  $\nu$ such that
\be E\left(\sup_{0\leq s\leq t}\p\un(s)\p^{p}\right)\leq C,\label{upp}\ee
\end{lem}
\begin{proof}
Since $\xn=\nabla\wedge\un$.
We apply the {\rm curl} to the equation \eqref{A},  we get for $t\in [0,T]$
\be d\xn + \nu A\xn dt + \nabla\wedge B(\un,\un) dt= 
\nabla\wedge f dt +  \sum_{i=1}^{\infty}\nabla\wedge 
(C^{i}\un)d\beta^{i}(t)\label{8}.\ee

Now, similar computations using the It\^o formula and then the elliptic system \eqref{24} yields the 
required estimate.
\end{proof}

\subsubsection{Tightness and the limit problem }

\begin{prp}
The family $\left\{{\cal L}(\un)\right\}_{\nu}$ is tight in $L^{2}(0,T; H)
\cap C([0,T]; D(A^{-\sigma/2}))$, for some $\sigma>1$.
\end{prp}
\begin{proof}
We decompose $\un$ as 
\beq
\un(t)&=&P_{n}u_{0}-\nu\int_{0}^{t}A\un(s)
-\int_{0}^{t}B(\un(s),\un(s))\nonumber\\
& &+\int_{0}^{t}f(s) 
+\sum_{i=1}^{\infty}\int_{0}^{t}C^{i}\un(s)d\beta^{i}(s)\nonumber\\
&=& J_{1}+...+J_{5}.\label{12}
\eeq
In the sequel, $C$ denotes and arbitrary positive constant independent of $\nu$. 
We have 
$$E|J_{1}|^{2}\leq C.$$

From \eqref{upp} 
$$E\p J_{2}\p^{2}_{W^{1,2}(0,T;V')}\leq C,$$
$$E\p J_{4}\p^{2}_{W^{1,2}(0,T;V')}\leq C.$$ 

Using Lemma \ref{integral},  assumption (G1)' and 
the estimate \eqref{up} we have
$$E\p J_{5}\p^{2}_{W^{\gamma,2}(0,T;H)}\leq C$$
for $\gamma\in(0,1/2)$.

Since $\alpha>1$, $D(A^{\alpha/2})\subset (L^{\infty}(D))^{2}$ so that
$$|<B(u,u),v>|\leq C|u|\p u\p|A^{\alpha/2}v|,\ u\in V,
\ v\in D(A^{\alpha/2})$$
for some constant $C>0$. Hence, we have
$$\p J_{3}\p^{2}_{W^{1,2}(0,T;D(A^{-\alpha/2}))}\leq C
\sup_{0\leq t\leq T}|\un(t)|^{2}\int_{0}^{T}\p\un(s)\p^{2}ds.$$

In virtue of \eqref{up} and \eqref{upp}, we obtain that
$$E\p J_{3}\p^{2}_{W^{1,2}(0,T;D(A^{-\alpha/2}))}\leq C.$$
Clearly for $\gamma\in(0,1/2)$, 
$W^{1,2}(0,T;D(A^{-\alpha/2}))\subset W^{\gamma,2}(0,T;D(A^{-\alpha/2}))$.
Collecting all the previous inequalities we get that
\be E\p\un\p_{W^{\gamma,2}(0,T;D(A^{-\alpha/2}))}\leq C,\label{usob}\ee
for $\gamma\in(0,1/2)$ and $\alpha>1$.

By \eqref{upp} and \eqref{usob}, we have that the laws  of $\un$ denoted by ${\cal L}(\un)$ 
are bounded in probability in 
$$L^{2}(0,T;V)\cap W^{\gamma,2}(0,T;D(A^{-\alpha/2})).$$
Using Lemmas \ref{compact1} and \ref{compact2}, we deduce that 
$\left\{{\cal L}(\un)\right\}$ is tight in 
$L^{2}(0,T;H)\cap C([0,T];D(A^{-\sigma/2}))$ for $\sigma>\alpha$.
\end{proof}

We conclude the existence of martingale solutions for system \eqref{abstract} 
by using the Skorohod theorem and a representation theorem for martingales.

\section{Stochastic Euler equation with multiplicative noise in Banach spaces}\label{BP}

The results of this section are due to \cite{BP01} where the techniques used are similar to 
Section \ref{weak}. For simplicity, we will assume that $f=0$.
Let us assume through this section that 
$G$ is a  continuous mapping 
from $H^{1,2}\cap H^{1,q}$ into $L_{2}(K; W^{1,2})\cap R(K; W^{1,q})$ such that

\begin{equation}\label{g1}
\|G(u)\|_{L_{2}(K; W^{1,2})}\leq C(1+\|u\|)
\end{equation}
and 
\begin{equation}\label{gq}
\|G(u)\|_{R(K; W^{1,q})}\leq C(1+\|u\|_{H^{1,q}}).
\end{equation}

\begin{definition}
Assume that $u_{0}\in H^{1,2}\cap H^{1,q}\quad {\rm for}\quad q\in[2,\infty)$ and that $G$ is a continuous mapping from $H^{1,2}\cap H^{1,q}$ into $L_{2}(K; W^{1,2})\cap R(K; W^{1,q})$.

A martingale $H^{1,2}\cap H^{1,q}$-valued solution to the stochastic Euler equation \eqref{SEE} 
is a triple consisting of a filtered probability space $(\Omega,{\cal F}, \left\{{\cal F}\right\}_{t\in [0,T]}, P)$, and 
${\cal F}_{t}$-adapted cylindrical Wiener process $W(t), t\geq 0$ on $K$ and 
an ${\cal F}_{t}$-adapted measurable  $H^{1,2}\cap H^{1,q}$-valued process $u(t), t\geq 0$ such that

\begin{enumerate}
\item for every $p\in [1,\infty),\quad u\in L^{p}(\Omega; L^{\infty}(0,T;  H^{1,2}\cap H^{1,q}))$,
\item for all $\phi\in D(A)$ and $t\in [0,T]$, one has a.s.
\begin{equation*}
<u(t),\phi>=<u_{0},\phi>+\int_{0}^{t}<u(s)\cdot\nabla\phi, u(s)>ds+\int_{0}^{t}<G(u(s))dW,\phi>.
\end{equation*}
\end{enumerate}
\end{definition}

\begin{thm} Let $q\in [2,\infty)$ and assume that the mapping $G$ defined previously satisfies the assumption \eqref{g1} and \eqref{gq}. Then for any $u_{0}\in H^{1,2}\cap H^{1,q}$ there exists a 
martingale $H^{1,2}\cap H^{1,q}$-valued solution to the system \eqref{SEE}.
\end{thm}
\begin{proof} See \cite{BP01}. 
\end{proof}

\section{Appendix}

For any Progressively measurable process 
$f\in L^{2}(\Omega\times[0,T]; L_{2}(K,H))$ denote by $I(f)$ the Ito 
integral defined as
$$I(f)(t) = \int_{0}^{t}f(s)dw(s),\ t\in [0,T].$$
$I(f)$ is a progressively measurable process in $L^{2}(\Omega\times[0,T];H)$.

\begin{lem}\label{integral}
Let $p\geq 2$ and $\gamma<1/2$ be given. Then for any progressively 
measurable process $f\in L^{2}(\Omega\times[0,T]; L_{2}(K,H))$, we have
$$I(f)\in L^{p}(\Omega; W^{\gamma,2}(0,T; H))$$
and there exists a constant $C(p,\gamma)>0$ independent of $f$ such that

$$E\p I(f)\p^{p}_{W^{\gamma,2}(0,T; H)}\leq C(p,\gamma)
E\int_{0}^{T}\p f\p^{p}_{L_{2}(K; H)}dt.$$
\end{lem}
\begin{proof} see \cite{FG95}.
\end{proof}

\begin{thm}\label{compact1}
Let $B_{0}\subset B\subset B_{1}$ be Banach spaces, $B_{0}$ and $B_{1}$ 
reflexive with compact embedding of $B_{0}$ in $B_{1}$. Let $p\in (1,\infty)$
and $\gamma\in(0,1)$ be given. Let $X$ be the space
$$X=L^{p}(0,T; B_{0})\cap W^{\gamma,2}(0,T; B_{1})$$
endowed with the natural norm. Then the embedding of $X$ in
$L^{p}(0,T; B_{0})$ is compact.
\end{thm}
\begin{proof}\cite{FG95}.
\end{proof}

\begin{thm}\label{compact2}
Let $B_{1}$ and $\tilde{B}$ two Banach spaces such that 
$B_{1}\subset\tilde{B}$ with compact embedding. If the real numbers 
$\gamma\in(0,1)$ and $p>1$ satisfy
$$\gamma p>1$$
then the space $W^{\gamma,2}(0,T; B_{1})$ is compactly embedded into
$C([0,T]; {\tilde B})$.
\end{thm}
\begin{proof}\cite{FG95}.
\end{proof}

\begin{lem}\label{parab} 
Let $v_{0}\in H_{0}^{1}(D)$ and $g\in L^{2}\left((0,T)\times D\right)$ and assume that 
$u$ is given such that $u\in L^{2}(0,T; [H^{1}(D)]^{2})$. Then 
the following equation

\be\left\{\begin{array}{lr}
                \frac{\pat v}{\pat t} + (u\cdot\nabla)v
                 = \nu\Delta v + g, &{\rm in}\ (0,T)\times D\\
                 v = 0, &{\rm on}\ (0,T)\times \pat D\\
                 v|_{t=0} = v_{0},&{\rm in}\ D
           \end{array} \right. \label{38}\ee
has a unique solution
$$v\in C([0,T]; H_{0}^{1}(D))\cap L^{2}(0,T; H^{2}(D))
\cap H^{1}(0,T; L^{2}(D)).$$
Assuming only $v_{0}\in L^{2}(D)$ and $g\in L^{2}(0,T;H^{-1}(D))$, it has 
a unique solution 
$$v\in C([0,T]; L^{2}(D))\cap L^{2}(0,T; H_{0}^{1}(D))
\cap H^{1}(0,T; H^{-1}(D)).$$
\end{lem}
\begin{proof}
{\bf Step 1.} We have the following a priori estimates:
\beq
\frac{1}{2}\frac{d}{dt}|v|_{L^{2}(D)}^{2} &=& \int_{D}(\nu\Delta v + g 
- (u\cdot\nabla)v)v
\nonumber\\
&\leq& -\nu|v|_{H^{1}(D)}^{2} + \frac{\nu}{4}|v|_{H^{1}(D)}^{2}
+ C_{0}|g|_{H^{-1}(D)}^{2}\nonumber\\ 
& &+ C_{1}|v|_{H^{1}(D)}|v|_{L^{4}(D)}|u|_{[L^{4}(D)]^{2}}\nonumber
\eeq
and
\beq
|v|_{H^{1}(D)}|v|_{L^{4}(D)}|u|_{[L^{4}(D)]^{2}}&\leq& 
\frac{\nu}{4}|v|_{H^{1}(D)}^{2}\nonumber\\ 
& &+ C_{2}|v|_{L^{2}(D)}|v|_{H^{1}(D)}|u|_{[L^{2}(D)]^{2}}|u|_{[H^{1}(D)]^{2}}
\nonumber\\
&\leq& \frac{\nu}{2}|v|_{H^{1}(D)}^{2}
+ C_{3}|v|_{L^{2}(D)}^{2}|u|_{[L^{2}(D)]^{2}}|u|_{[H^{1}(D)]^{2}}
\nonumber
\eeq
so that

$$\frac{1}{2}\frac{d}{dt}|v|_{L^{2}(D)}^{2}
- \frac{\nu}{4}|v|_{H^{1}(D)}^{2}\leq
C_{0}|g|_{H^{-1}(D)}^{2} 
+ C_{4}|v|_{L^{2}(D)}^{2}|u|_{[L^{2}(D)]^{2}}|u|_{[H^{1}(D)]^{2}}.$$

Whence (by Gronwall lemma and again by the same inequality, using the 
regularity of $u$ which implies that $|u|_{[L^{2}(D)]^{2}}|u|_{[H^{1}(D)]^{2}}
\in L^{1}(0,T)$)

$$\sup_{t\in [0,T]}|v(t)|_{L^{2}(D)}^{2} < \infty,
\ \ \ \ \int_{0}^{T}|v(s)|_{H^{1}(D)}^{2}ds <\infty.$$

Proving these estimates for classical Galerkin approximation and passing to
the limit in the classical way, we prove that there exists a solution
$$v\in C([0,T]; L^{2}(D))\cap L^{2}(0,T; H_{0}^{1}(D))
\cap H^{1}(0,T; H^{-1}(D)).$$
The uniqueness is proved by the very similar estimates.\\
{\bf Step 2.} We have the following additional a priori estimate
\beq
\frac{1}{2}\frac{d}{dt}|\nabla v|_{L^{2}(D)}^{2} &=& 
- \int_{D}(\nu\Delta v + g - (u\cdot\nabla)v)\Delta v
\nonumber\\
&\leq& -\nu|\Delta v|_{L^{2}(D)}^{2} 
+ \frac{\nu}{4}|\Delta v|_{L^{2}(D)}^{2}
+ C_{5}|g|_{L^{2}(D)}^{2}\nonumber\\
&+& C_{6}|\Delta v|_{L^{2}(D)}|v|_{W^{1,4}(D)}|u|_{[L^{4}(D)]^{2}}\nonumber
\eeq
and
\beq
|\Delta v||v|_{W^{1,4}(D)}|u|_{[L^{4}(D)]^{2}}&\leq&
\frac{\nu}{4}|\Delta v|_{L^{2}(D)}^{2}\nonumber\\ 
&+& C_{7}|v|_{H^{1}(D)}|v|_{H^{2}(D)}|u|_{[L^{2}(D)]^{2}}|u|_{[H^{1}(D)]^{2}}
\nonumber\\
&\leq&\frac{\nu}{2}|\Delta v|_{L^{2}(D)}^{2}
+ C_{8}|v|_{H^{1}(D)}|u|_{[L^{2}(D)]^{2}}|u|_{[H^{1}(D)]^{2}}
\nonumber
\eeq
so that
$$\frac{1}{2}\frac{d}{dt}|\nabla v|_{L^{2}(D)}^{2} 
- \frac{\nu}{4}|\Delta v|_{L^{2}(D)}^{2}\leq
C_{5}|g|_{L^{2}(D)}^{2} 
+ C_{9}|\nabla v|_{L^{2}(D)}|u|_{[L^{2}(D)]^{2}}|u|_{[H^{1}(D)]^{2}}.$$
Whence (as in step 1)

$$\sup_{[0,T]}|\nabla v|_{[L^{2}(D)]^{2}}^{2} < \infty,
\ \ \ \ \int_{0}^{T}|\Delta v|_{L^{2}(D)}^{2}ds <\infty.$$
Proving these estimates for classical Galerkin approximations and passing 
to the limit in the classical way, we prove that the solution of step 1
satisfies the regularity required by the lemma.
\end{proof}

\begin{lem}\label{mes}
Let $X$ and $Y$ be two separable Banach spaces and $\Lambda$ a multiple-valued
mapping from $X$ to the set of nonempty closed subsets of $Y$, the graph of 
$\Lambda$ being closed.\\
Then $\Lambda$ admits a universally Radon measurable section, i.e., there 
exists a mapping $L$ from $X$ to $Y$, such that
$$L(x)\in\Lambda(x)\ \ \ \forall x\in X,$$
and $L$ is measurable for any Radon measure defined on the Borel sets of $X$.
\end{lem}

{\bf Acknowledgements}: Hakima Bessaih is extremely grateful to the Bernoulli center in Lausanne where these notes have started as a short course and to the IMA where most of these notes have been written. Moreover, she would like to thank the anonymous referee for pointing out a few mistakes in the preliminary version of the paper. 


\begin{thebibliography}{99}
\bibitem {bardos}
C. Bardos (1972), {\it \'{E}xistence et unicit\'{e} de la solution de 
l'\'{e}quation d'Euler en dimensions deux}, Jour. Math. Anal. Appl. 
{\bf 40},  769-780.

\bibitem{BT73}
A. Bensoussan, R. Temam (1973), {\it Equations stochastiques du type
Navier-Stokes}, J. Funct. Anal, {\bf13}, 195-222.

\bibitem{BF99}
H. Bessaih, F. Flandoli (1999) , {\it 2-D Euler equation perturbed by noise}, 
NoDEA, {\bf 6}, 35--45.

\bibitem{B99}
H. Bessaih (1999), {\it Martingale solutions for stochastic Euler equations}, Stochastic Analysis and Applications, {\bf 17} (5), 713--725.

\bibitem{B00}
H. Bessaih (2000), {\it Stochastic weak attractor for a dissipative Euler equation},  Electron. J. Probab., 
{\bf 5}, no. 3, 1--16.

\bibitem{B08}
H. Bessaih (2008), {\it Stationary solutions for the 2D stochastic dissipative Euler equation},  
Seminar on Stochastic Analysis, Random Fields and Applications V, 23Ð36, Progr. Probab., 59, BirkhŠuser, Basel. 

\bibitem {BP01}
Z. Brzezniak, S. Peszat (2001), {\it Stochastic two dimensional Euler equations}, The annales of Probability, {\bf 29} No 4, 1796--1832.

\bibitem{CM}
I. Chueshov, A. Millet (2010), {Stochastic 2D hydrodynamical type systems: Well posedness
 and large deviations}, \emph{Appl. Math.  Optim.}, \textbf{61-3}, 379--420. 
 
\bibitem {DZ92}
G. Da Prato, J. Zabczyk (1992), {\it Stochastic Equations in Infinite
Dimensions}, Cambridge university press.

\bibitem{FG95}
F. Flandoli, D. Gatarek (1995), {\it Martingale and stationary solutions
for stochastic Navier-Stokes equations}, Probab. Theory relat. Fields, 
{\bf 102}, 367-391.

\bibitem{GV}
N. E. Glatt-Holtz, V. C. Vicot (2011), {\it Local and global existence of smooth solutions for the stochastic Euler equations with multiplicative noise}, to appear in Ann. Prob.

\bibitem{kato}
T. Kato (1992), {\it A remark on a theorem of C.Bardos on the 2D-Euler
equation}, Preprint.

\bibitem{K02}
J. U. Kim (2002), {\it On the stochastic Euler equations in a two-dimensional domain},
SIAM J. Math. Anal., {\bf 33} (5), 1211--1227.

\bibitem{K09}
J. U. Kim (2009), {\it Existence of a local smooth solution in Probability to the stochastic Euler equations in $\R^{d}$}, Jour. of Func. Anal., {\bf 256}, 3660--3687.

\bibitem {Lions61}
J.L. Lions (1961), {\it \'{E}quations Diff\'{e}rentielles 
Op\'{e}rationelles et Probl\`{e}mes aux limites}, Springer-Verlag, Berlin,
1961.
 
\bibitem {PLions96}
P.L. Lions (1996), {\it Mathematical Topics in Fluid Mechanics}, 
Vol 1, Incompressible Models, Oxford Sci. Publ, Oxford, 1996.

\bibitem {MV}
R. Mikulevicius, G. Valiukevicius (2000), {\it On stochastic Euler equations in $\R^{d}$},
Electron. J. Probab., {\bf 5} (6), 1--20.
 
\bibitem {temam84} 
R. Temam (1984), {\it Navier-Stokes equations}, North-Holland.


\end{thebibliography}
\end{document}